\documentclass[11pt]{amsart}

\usepackage[utf8]{inputenc}

\usepackage{amsfonts, amstext, amsmath, amsthm, amscd, amssymb}
\usepackage{manfnt}
\usepackage{mathrsfs}
\usepackage{tikz}
\usepackage{url}
\usepackage[all]{xypic}

\usepackage{color}
\usepackage[colorlinks,
    linkcolor={red!50!black},
    citecolor={blue!50!black},
    urlcolor={blue!80!black}]{hyperref}

\usepackage{enumerate}

\renewcommand{\setminus}{{\smallsetminus}}

\newcommand{\bp}{\begin{pmatrix}}
\newcommand{\ep}{\end{pmatrix}}
\newcommand{\be}{\begin{equation}}
\newcommand{\ee}{\end{equation}}
\newcommand{\ol}[1]{\overline{#1}}

\newcommand{\PD}{\operatorname{PD}}

\newcommand{\smfrac}[2]{\mbox{\footnotesize$\displaystyle\frac{#1}{#2}$}} % small medium frac
\numberwithin{equation}{section}

\theoremstyle{plain}
\newtheorem{theorem}[equation]{Theorem}
\newtheorem{lemma}[equation]{Lemma}
\newtheorem{proposition}[equation]{Proposition}

\theoremstyle{definition}
\newtheorem{example}[equation]{Example}

\newtheorem{definition}[equation]{Definition}

\theoremstyle{remark}
\newtheorem*{claim}{Claim}
\newtheorem{remark}[equation]{Remark}

\numberwithin{equation}{section}

\def\Z{\mathbb Z}

\def\Q{\mathbb Q}
\def\C{\mathbb C}

\def\wt#1{\widetilde{#1}}

\def\sm{\setminus}

\def\a{\alpha}

\def\toiso{\xrightarrow{\cong}}

\def\bp{\begin{pmatrix}}
\def\ep{\end{pmatrix}}
\def\ba{\begin{array}}
\def\ea{\end{array}}
\def\bn{\begin{enumerate}}
\def\en{\end{enumerate}}

\DeclareMathOperator\lk{lk}

\DeclareMathOperator\Tor{Tor}
\DeclareMathOperator\Hom{Hom}

\DeclareMathOperator\Id{Id}

\DeclareMathOperator\im{im}

\DeclareMathOperator\pt{pt}

\newcommand{\eps}{\varepsilon}
\newcommand{\transpose}{T}

\begin{document}

\title{Concordance invariance of Levine-Tristram signatures of links}

\author{Matthias Nagel}
\address{D\'epartement de Math\'ematiques,
Universit\'e du Qu\'ebec \`a Montr\'eal, Canada}
\email{nagel@cirget.ca}

\author{Mark Powell}
\address{D\'epartement de Math\'ematiques,
Universit\'e du Qu\'ebec \`a Montr\'eal, Canada}
\email{mark@cirget.ca}

%\date{\today}
%\makeatletter
%\@namedef{subjclassname@2010}{\textup{2010} Mathematics Subject Classification}
%\makeatother
%\subjclass[2010]{primary: , secondary: }
%\keywords{embedded Morse theory, manifold with boundary, cobordism, critical points}

\def\subjclassname{\textup{2010} Mathematics Subject Classification}
\expandafter\let\csname subjclassname@1991\endcsname=\subjclassname
\expandafter\let\csname subjclassname@2000\endcsname=\subjclassname

\begin{abstract}
We determine for which complex numbers on the unit circle the Levine-Tristram
signature and the nullity give rise to link concordance invariants.
\end{abstract}

\maketitle

\section{Introduction}

Let $L \subset S^3$ be an $m$-component oriented link in the 3-sphere.
Each connected, oriented Seifert surface~$F$ for $L$ has a bilinear Seifert form defined by
\begin{align*}
V \colon H_1(F;\Z) \times H_1(F;\Z) &\to \Z\\
(p[x],q[y]) &\mapsto pq\lk(x^-, y),
\end{align*}
where $p,q \in \Z$,
$x,y$ are simple closed curves on $F$ with associated homology classes $[x],[y]$, and $x^-$ is a push-off of $x$ in the negative normal direction of~$F$.
Given a unit modulus complex number $z \in S^1 \sm \{1\}$, choose a basis for $H_1(F;\Z)$
and define the hermitian matrix
\[B(z) := (1-z)V + (1-\ol{z})V^\transpose.\]
The \emph{Levine-Tristram signature} $\sigma_L(z)$ of $L$ at $z$ is defined to be
the signature of $B(z)$, namely the number of positive eigenvalues minus the number of negative eigenvalues.
The \emph{nullity}~$\eta_L(z)$
of $L$ at $z$ is the dimension of the null space of $B(z)$.
Both quantities can be shown to be invariants of the
$S$-equivalence class of the Seifert matrix, and are therefore
link invariants~\cite{Levine:1969-2,Tristram:1969-1}.

We say that two oriented $m$-component links $L$ and $J$ are \emph{concordant}
if there is a flat embedding into $S^3 \times I$ of a disjoint union of~$m$
annuli $A \subset S^3 \times I$, such that the oriented boundary of $A$
satisfies
\[ \partial A = -L \sqcup J \subset -S^3 \sqcup S^3 = \partial (S^3 \times I). \]
An $m$-component link $L$ is \emph{slice} if it is concordant to the $m$-component unlink.

The purpose of this paper is to answer the following question: \emph{for which values of $z$ are $\sigma_L(z)$ and $\eta_L(z)$ link concordance invariants?}
We work in the topological category, in order to obtain the strongest possible results.
In order to state our main theorem, we need one more definition.

\begin{definition}\label{definition:knotnullstelle}
A complex number~$z \in S^1 \sm \{1\}$ is a \emph{Knotennullstelle} if there exists a Laurent polynomial~$p(t) \in \Z[t,t^{-1}]$ with $p(1)=\pm 1$ and $p(z)=0$.
\end{definition}

Note that a complex number $z \in S^1 \sm \{1\}$ is a Knotennullstelle if and only if
there exists a knot~$K$ whose Alexander polynomial $\Delta_K$ has the
property that $\Delta_K(z) = 0$. This follows from the fact that all Laurent polynomials
$q \in \Z[t,t^{-1}]$ with $q(1) = \pm 1$ and $q(t) = q(t^{-1})$ can be realised
as Alexander polynomials of knots \cite[Theorem~8.13]{Burde-Zieschang:1985-1}.
Here is our main theorem.

\begin{theorem}\label{theorem:main-theorem}
The link invariants $\sigma_L(z)$ and $\eta_L(z)$ are concordance invariants
if and only if $z \in S^1 \sm \{1\}$ does not arise as a Knotennullstelle.
\end{theorem}

\subsubsection*{Discussion of previously known results}

The first point to note is that, due to J.~C.~Cha and C.~Livingston~\cite{Cha-Livingston:2002-1}, when $z$ is a Knotennullstelle neither $\sigma_L(z)$ nor $\eta_L(z)$ are link concordance invariants.

\begin{theorem}[Cha, Livingston]\label{theorem:cha-livingston}
For any Knotennullstelle $z \in S^1 \sm\{1\}$, there exists a slice knot $K$ with $\sigma_K(z) \neq 0$ and $\eta_K(z) \neq 0$.
\end{theorem}

Given a polynomial $p(t)$ with $p(1)=\pm 1$ and $p(z)=0$, Cha and Livingston construct a matrix $V$ with $V-V^\transpose$ nonsingular, with $\det(tV-V^\transpose)$ equal to $p(t)p(t^{-1})$, such that the upper left half-size block contains only zeroes, and such that $\sigma(B(z)) \neq 0$.  Such a matrix can easily be realised as the Seifert matrix of a slice knot.

Some positive results on concordance invariance are also known.  For $z$ a
prime power root of unity, $\sigma_L(z)$ and $\eta_L(z)$ are concordance
invariants; see \cite{Murasugi:1965-1},~\cite{Tristram:1969-1} and~\cite{Kauffman:1978}.
D.~Cimasoni and V.~Florens~\cite{Cimasoni-Florens-2008} dealt with multivariable
signature and nullity concordance invariants, but again only at prime power
roots of unity.

For the signature and nullity at algebraic numbers away from prime power roots of unity,
we could not find any statements or results in the literature pertaining to our question.
Levine~\cite{Levine:2007-1} studied the question in terms of $\rho$-invariants, but only discussed concordance invariance away from the roots of the Alexander polynomial.

By changing the rules slightly, one can obtain a concordance invariant for all $z$.
The usual method is to define a function that is the average of the two one-sided limits of the Levine-Tristram signature function.
Let $z= e^{i\theta} \in S^1$, and consider:
  \[\ol{\sigma}_L(z) := \smfrac{1}{2}\Big(\lim_{\omega \to \theta_+} \sigma(B(e^{i\omega})) + \lim_{\omega \to \theta_-} \sigma(B(e^{i\omega}))\Big).\]
Since prime power roots of unity are dense in $S^1$, this averaged signature function yields a concordance invariant at every $z \in S^1$.  The earliest explicit observation of this that we could find was by Gordon in the survey article~\cite{Gordon:1978-1}.
One can also consider the averaged nullity function, to which similar remarks apply:
\[\ol{\eta}_L(z) := \smfrac{1}{2}\Big(\lim_{\omega \to \theta_+} \eta(B(e^{i\omega})) + \lim_{\omega \to \theta_-} \eta(B(e^{i\omega}))\Big).\]
In particular this is also a link concordance invariant.

Note that the function $\sigma_L\colon S^1 \sm \{1\} \to \Z$ is continuous away from roots of the Alexander polynomial $\det(tV-V^T)$ of $L$.  More generally one can consider the torsion Alexander polynomial $\Delta_L^{\Tor}$~of $L$, which by definition is the greatest common divisor of the $(n-r) \times (n-r)$ minors of $tV-V^T$, where $n$ is the size of $V$ and $r$ is the minimal nonnegative integer for which the set of minors contains a nonzero polynomial.
The function $\sigma_L$ is continuous away from the roots of the torsion Alexander polynomial~$\Delta_L^{\Tor}$, by
\cite[Theorem~2.1]{Gilmer-Livingston-2015} (their $A_L$ is our $\Delta_L^{\Tor}$).

Thus if $z$ is not a root of the torsion Alexander polynomial of \emph{any} link, the signature cannot jump at that value, and the signature function $\sigma_L(z)$ equals the averaged signature function $\ol{\sigma}_L(z)$ there.  Since the averaged function is known to be a concordance invariant, the non-averaged function is also an invariant when $z$ is not the root of any link's Alexander polynomial.  The excitement happens when $z$ is the root of the Alexander polynomial of some link, but is not the root of an Alexander polynomial of any knot.  The averaged and non-averaged signature functions can differ at such~$z$, but nevertheless \emph{both} are concordance invariants.  In Section~\ref{section:application} we will give an example which illustrates this difference, and gives an instance where the non-averaged function is more powerful.  Similar examples were given in \cite{Gilmer-Livingston-2015}, but only with jumps occurring at prime power roots of unity.

Finally we remark that our proof of Theorem~\ref{theorem:main-theorem} covers the previously known cases of prime power roots of unity and transcendental numbers, as well as the new cases.

\subsubsection*{Organisation of the paper}
The rest of the paper is organised as follows.
In Section~\ref{section:application}, we give an example of two links that are not concordant, where we use the signature and nullity functions at a root of their Alexander polynomials, which is not a prime power root of unity, to detect this fact.
Section~\ref{section:concordance-invariance-nullity} proves that the nullity is a concordance invariant, and the corresponding fact for signatures is proven in Sections~\ref{section:identification-signatures} and~\ref{section:conc-invariance-signature}.

\subsubsection*{Acknowledgements}
We thank Enrico Toffoli for pointing out a mistake in Lemma~\ref{lemma:identification} in a previous version.
We thank Stefan Friedl, Pat Gilmer, Chuck Livingston and Andrew
Ranicki for helpful discussions.
In particular~\cite{Gilmer-Livingston-2015}
inspired the question that led to this paper.
We also thank the referee for helpful feedback.
M.\ Powell is supported by an
NSERC Discovery grant. M.\ Nagel is supported by a CIRGET Postdoctoral
Fellowship.

\section{An application}\label{section:application}

In the introduction, for a link~$L$ we defined the signature function~$\sigma_L(z)$ and the
nullity function~$\eta_L(z)$, for each $z \in S^1\sm\{1\}$.
From the characterisation in Theorem~\ref{theorem:main-theorem}, one
easily finds new values~$z$ for which it was not previously known that
$\sigma(z)$ and $\eta(z)$ are concordance invariants.
In Proposition~\ref{prop:PropertiesOfTheLinks}, by exhibiting the obligatory
explicit example, we show that these values give obstructions to concordance
that are independent from previously known obstructions coming from the
signature and nullity functions.
We finish the section by constructing, in Proposition~\ref{prop:UniversalExample},
a family of such examples for any algebraic number on $S^1$.

Before the construction, we collect some facts on the set of roots of Alexander polynomials
of links. We say that a complex number~$z \in S^1 \sm \{1\}$ is a \emph{Linknullstelle} if $z$ is a root
 of a non-vanishing single variable Alexander polynomial of some link.
We have the following inclusions:
\[ \ba{rclcl} \left\{\mbox{Knotennullstellen}\right\}
&\subset&
\left\{\mbox{Linknullstellen}\right\}
&\subset &
S^1 \sm \{1\}  \\
& &\hspace{3em}\cup & & \\
&& \left\{\ba{c}\mbox{prime power}\\\mbox{roots of $1$}\ea\right\} &&
\ea \]
We will see that these inclusions are strict. The two subsets of the set of Linknullstellen are disjoint, since no prime power root of unity can be a root of a polynomial that augments to $\pm 1$, because the corresponding cyclotomic polynomial augments to the prime.
Moreover, the union of the Knotennullstellen and the prime power roots of unity is not exhaustive.

\begin{lemma}\label{lem:Nullstellen}~
\begin{enumerate}
\item\label{item:nullstellen-lemma-one} The set of Linknullstellen coincides with the set of algebraic numbers in $S^1 \setminus \{1\}$.
\item\label{item:nullstellen-lemma-two} The number~$z_0 = \frac{3+4i}{5} \in S^1$ is an algebraic number, which is neither a Knotennullstelle nor
a root of unity.
\end{enumerate}
\end{lemma}

\begin{proof}
Let $z \in S^1 \sm \{1\}$ be an algebraic number, so that $p(z)=0$ for some $p \in \Z[t]$.  Let
\[q(t):=(t-1)^3 p(t) p(t^{-1}) \in \Z[t,t^{-1}].\]
We claim that there is a link $L$ with single variable Alexander polynomial $\Delta_L(t) = q(t)$.
Choose a $2$-variable polynomial $P(x,y) \in \Z[x^{\pm 1},y^{\pm 1}]$ with $P(t,t)=p(t)$.  Let \[Q(x,y) := (x-1)(y-1)P(x,y)P(x^{-1},y^{-1}).\]
A corollary~\cite[Corollary~8.4.1]{Hillman:2012-1-second-ed} to Bailey's
theorem~\cite{Bailey-1977} states that any polynomial $Q(x,y)$ in
$\Z[x^{\pm 1},y^{\pm 1}]$, with $Q = \ol{Q}$ up to multiplication by $\pm x^k y^{\ell}$,
and such that $(x-1)(y-1)$ divides $Q$, is the Alexander polynomial of some
$2$-component link of linking number zero.
Thus there exists a $2$-component link $L$ with $2$-variable Alexander polynomial $Q(x,y)$.

The single variable Alexander
polynomial $\Delta_L(t)$ is obtained from the $2$-variable Alexander polynomial of
a $2$-component link $Q(x,y)$ as $(t-1)Q(t,t)$ \cite[Remark 9.18]{Burde-Zieschang:1985-1}.
But
\[(t-1)Q(t,t) = (t-1)^3P(t,t)P(t^{-1},t^{-1}) = (t-1)^3 p(t) p(t^{-1}) = q(t).\]
This completes the proof of the claim and therefore of (\ref{item:nullstellen-lemma-one}): the set of Linknullstellen is the set of
algebraic numbers lying on $S^1 \sm \{1\}$.

%Also, given any polynomial $q(x,y)$,
%one can always produce a polynomial satisfying the above requirements as
%$(x-1)(y-1)q(x,y)q(x^{-1},y^{-1})$.  Thus the set of Linknullstellen is the set of
%algebraic numbers lying on $S^1 \sm \{1\}$.

For (\ref{item:nullstellen-lemma-two}), first observe that the complex number~$z_0 := \frac{3+4i}{5}$ has unit modulus and that $z_0$ is a zero of the polynomial
\[ p(t) := 5 t^2 -6 t + 5,\]
and therefore is an algebraic number. Note that no cyclotomic polynomial divides the polynomial~$p(t)$. This
can be checked for the first six by hand, and the rest have degree larger than $2$.
From Abel's irreducibility theorem, we learn that $z_0$ is not a zero of a cyclotomic polynomial
and thus is not a root of unity.  Since $p(1)=4$ and $p(t)$ is irreducible over $\Z[t]$,
$z_0$ is not the root of any polynomial that augments to $\pm 1$. As a result,
$z_0$ is not a Knotennullstelle.
\end{proof}

Next we describe links~$L$ and $L'$ whose signature and nullity functions are
equal everywhere on $S^1\sm \{1\}$ apart from at $z_0$, which will be a root of
the Alexander polynomials of $L$ and $L'$.
We find these links by realising suitable Seifert forms.
\begin{example}\label{example:SingleJump}
Consider the following Seifert matrix:
\[ V := \begin{pmatrix}
0 &  0 &  0 &  0 & 1 & 0 &  0 & 0 \\
0 &  0 &  0 &  0 & 0 & 5 & -4 & 4 \\
0 &  0 &  0 &  0 & 0 & 0 &  1 & 0 \\
0 &  0 &  0 &  0 & 0 & 0 &  0 & 1 \\
1 & -1 &  0 &  0 & 0 & 0 &  0 & 0 \\
0 &  5 & -1 &  0 & 0 & 0 &  0 & 0 \\
0 & -4 &  1 & -1 & 0 & 0 &  0 & 0 \\
0 &  4 &  0 &  1 & 0 & 0 &  0 & 1
\end{pmatrix}. \]
This matrix represents the Seifert form of the $3$-component link~$L$ given
by the boundary of the Seifert surface shown in Figure~\ref{Fig:Realisation}.
\begin{figure}[h]
  \begin{center}
	\begin{tikzpicture}
	\node[anchor=south west,inner sep=0] at (0,0) {\includegraphics[width=0.95\textwidth]{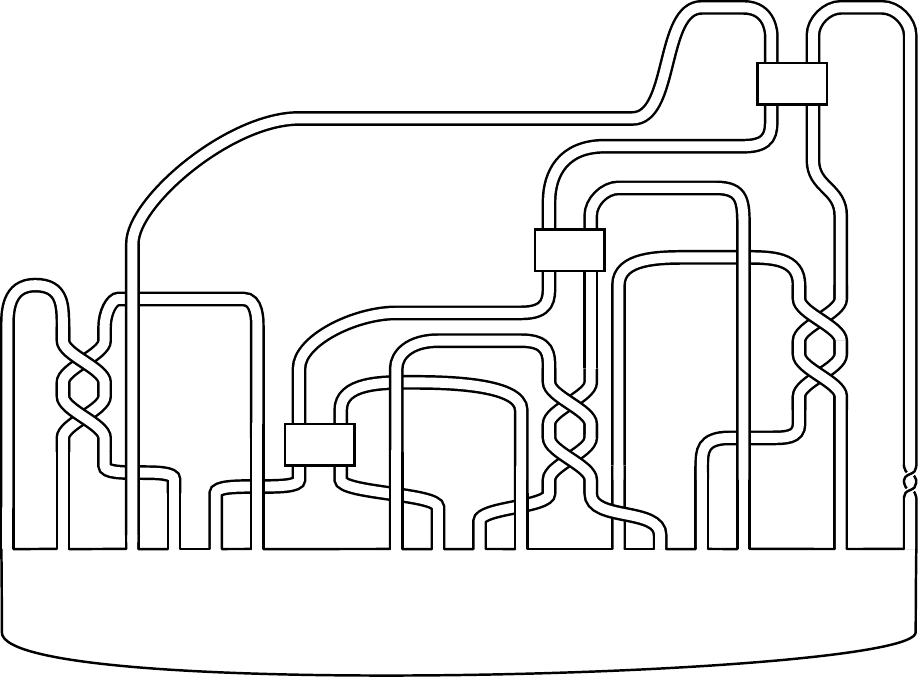}};
	%\draw[step=0.25, help lines] (0,0) grid (16,10);
\node at (0.3,1){$e_1$};
	\node[fill=white] at (1.75,1){$e_2$};
	\node[fill=white] at (2.25,1){$e_5$};
	\node[fill=white] at (0.25*21,1){$e_3$};
	\node[fill=white] at (0.25*23,1){$e_6$};
	\node[fill=white] at (0.25*33,1){$e_4$};
	\node[fill=white] at (0.25*35,1){$e_7$};
	\node[fill=white] at (0.25*44,1){$e_8$};
\node at (0.25*17-0.05,3.05){$-5$};
	\node at (0.25*27+0.70,5.6){$4$};
	\node at (0.25*41+0.15,7.8){$-4$};
	%\draw[fill] (1.7,2.2) circle (1pt) node[above] {$0$};
	%\draw[fill] (6.3,2.2) circle (1pt) node[above] {$1$};
	\end{tikzpicture}
  \end{center}
  \caption{Realisation of the Seifert form $V$.}
  \label{Fig:Realisation}
\end{figure}
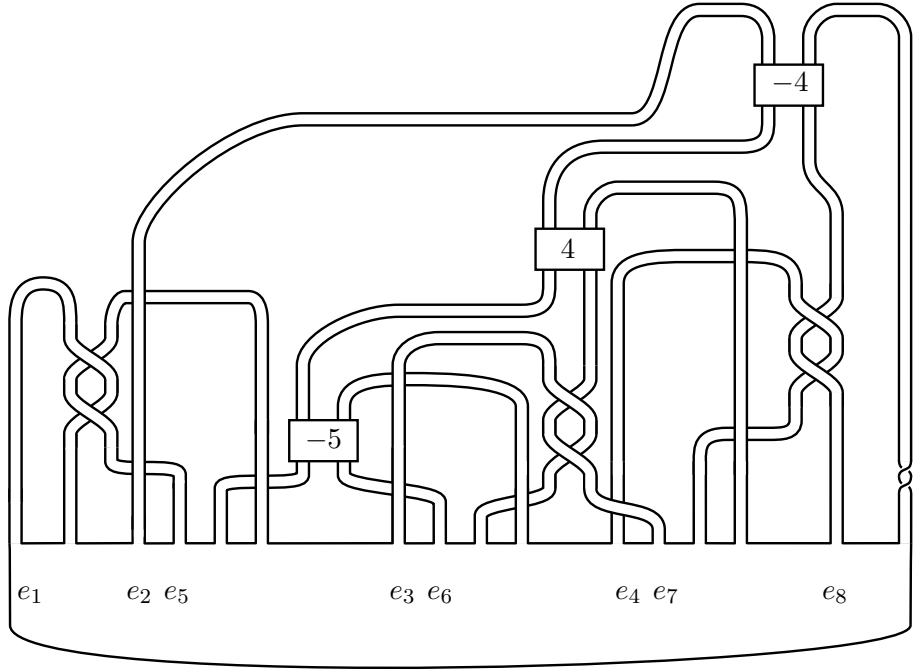
As usual, a box with $n \in \Z$ inside denotes $n$ full right-handed twists between two bands,
made without introducing any twists into the individual bands.  To see what we mean, observe that there are three
instances in the figure of one full left-handed twist, otherwise known as $-1$
full right-handed twists.  The left-most twist is between the bands labelled
$e_1$ and $e_5$.  To obtain the Seifert matrix, note that the beginning of each
of the eight bands is labelled $e_i$, for $i=1,\dots,8$.  Orient the bands
clockwise and compute using $V_{ij} = \lk(e_i^-, e_j)$, where the picture is
understood to show the positive side of the Seifert surface.

Produce a link~$L'$ from~$L$ by removing the single twist in the right-most band, labelled $e_8$ in Figure~\ref{Fig:Realisation}.
This gives rise to a Seifert matrix $V'$ for~$L'$ which is the same as~$V$,
except that the bottom right entry is a~$0$ instead of a~$1$.

Consider the sesquilinear form $B$ over $\Q[t^{\pm 1}]$ determined by the matrix
\[ (1-t) V + (1-t^{-1}) V^\transpose.\]
The form~$B$ splits into a direct sum of sesquilinear forms. For a Laurent polynomial
$p(t) \in \Q[t^{\pm 1}]$, abbreviate the form given by the $2\times 2$ matrix
\[ \begin{pmatrix}0 & p(t)\\ p(t^{-1}) & 0\end{pmatrix}. \]
by $\left[ p(t) \right]$.
A calculation shows that $B$ is congruent to the form
\[ [t-1] \oplus [t-1] \oplus [t-1] \oplus
\begin{pmatrix}
0 &  q(t)\\
q(t^{-1}) & -t^{-1} + 2 -t
\end{pmatrix}, \]
where the polynomial $q(t)$ is
%q(t) &=-5t^{-1} + 21 - 38t + 38t^{2} - 21t^{3} + 5t^{4} \\
\[ q(t)=t^{-1} \cdot (t -1 )^{3} \cdot (5t^2 - 6t + 5). \]

On the other hand the corresponding sesquilinear form $B'$ over $\Q[t^{\pm 1}]$ for $L'$ is equivalent to
\[ [t-1] \oplus [t-1] \oplus [t-1] \oplus [q(t)].\]
\end{example}

\begin{proposition}\label{prop:PropertiesOfTheLinks}
Let $z_0$ denote the algebraic number $\frac{3+4i}{5}$.
The links $L$ and $L'$ constructed in Example~\ref{example:SingleJump} have the following properties.
\begin{enumerate}
\item If $z$ is a root of unity, then
$\sigma_{L}(z) = \sigma_{L'}(z)$  and  $\eta_L(z) = \eta_{L'}(z)$.
\item The averaged signature and nullity functions agree, i.e.\
\[ \ol{\sigma}_L (z) = \ol{\sigma}_{L'} (z) \text{ and } \ol{\eta}_L (z) = \ol{\eta}_{L'} (z) \]
for all $z \in S^1\sm\{1\}$.
\item The signatures and nullities of $L$ and $L'$ at $z_0$ differ:
\[ \sigma_{L}(z_0) \neq \sigma_{L'}(z_0) \text{ and } \eta_{L}(z_0) \neq \eta_{L'}(z_0),\]
and so $L$ is not concordant to $L'$.
\end{enumerate}
\end{proposition}

\begin{proof}
Note that for any $z \in \C \setminus \{0,1\}$ with $q(z) \neq 0$,
the form $B(z)$ over $\C$ is nonsingular and metabolic. The same holds for $B'(z)$.
This implies that the signatures $\operatorname{sign} B(z)$ and $\operatorname{sign} B'(z)$ vanish.  The nullities $\eta_L(z),\eta_{L'}(z)$ are also both zero.
Since the roots of $q(z)$ are exactly
$z_0$ and $\ol{z_0}$, which are not roots of unity by Lemma~\ref{lem:Nullstellen},
we obtain the first statement of the proposition.
We also see that the averaged signature function on $S^1 \sm\{1\}$ and the averaged nullity function are identically zero, so we obtain the second statement.

From Lemma~\ref{lem:Nullstellen}, we know that $z_0 := \frac{3+4i}{5}$
is not a Knotennullstelle, and
\[ \operatorname{sign} B(z_0) = \operatorname{sign}
	\begin{pmatrix}
	0 & 0\\
	0 & \frac{4}{5}
	\end{pmatrix} = 1. \]
Thus $\sigma_L(z_0)=1 = \eta_L(z_0)$.  On the other hand, for $L'$ the matrix $B'(z_0)$ is a $2\times 2$ zero matrix, so we have that $\sigma_{L'}(z_0)=0$ and $\eta_{L'}(z_0)=2$.
Both signatures and the nullities at $z_0$ differ, so $L$ and $L'$ are not concordant by Theorem~\ref{theorem:main-theorem}.
\end{proof}

\begin{remark}
  One can also see that $L$ and $L'$ are not concordant using linking numbers.
\end{remark}

A more systematic study of the construction of the example above leads to the following
proposition.
\begin{proposition}\label{prop:UniversalExample}
Let $q(t)\in \Z[t]$ be a polynomial. Then there exists a natural number $k > 0$
and a link~$L$ with
Alexander polynomial $\Delta_L(t) \overset{.}{=} q(t^{-1}) q(t) (t-1)^k$ up to
units in $\Z[t,t^{-1}]$ such that
\begin{enumerate}
\item the form~$B(z)$ of $L$ is metabolic and nonsingular for all $z \in S^1\sm \{1\}$ which are not roots of $q(t)$, so $\sigma_L(z)=0$.
\item if $z_0\neq 1$ is a root of $q(t)$ of unit modulus, then $\sigma_L(z_0) \neq 0$.
\end{enumerate}
\end{proposition}

The proof of this proposition is based on ideas from~\cite{Cha-Livingston:2002-1}.

\begin{proof}
Consider the size $n+1$ square matrix $P$ with entries in $\Z[y]$ given by
\[ P(y):=
	\begin{pmatrix}
	1 & y & 0 && & ya_1\\
	0 & 1 & y && & ya_2\\
	\vdots && \ddots & \ddots & &\vdots\\
	&&& 1 & y    & ya_{n-1}\\
	0 &&& 0 & 1    & ya_n\\
	y & 0 & \ldots &0 & 0 & 0
	\end{pmatrix},\]
with $a_i$ integers.  Over $\Z[y^{\pm 1}]$, the matrix~$P$ can be transformed via invertible row operations
and column operations to
the matrix
\[ A(y) = \begin{pmatrix}
	1 & 0 & 0 && & p(y)\\
	0 & 1 & 0 && & 0\\
	\vdots && \ddots & \ddots & & \vdots\\
	&&& 1 & 0    & 0\\
	0 &&& 0 & 1    & 0\\
	y & 0 & \ldots &0 & 0 & 0
	\end{pmatrix}\]
with $p(y) = b_1(y)$ where $b_k(y) \in \Z[y]$ is defined by the
recursion~$b_{k-1}(y) := y\cdot (a_k - b_k(y))$ and $b_n(y) := y \cdot a_n$.
Notice that, up to units, we can arrange $p(y)$ to be any polynomial in
$\Z[y^{\pm 1}]$
by choosing $n$ sufficiently large and then suitable entries $a_k\in \Z$.  That is, multiply by $y^{\ell}$ so that the lowest order term is the linear term, and take $(-1)^i a_i$ to be the coefficient of $y^{i-1}$ in $p(y)$, for $i=2,\dots,n+1$.

Pick the entries $a_k$ so that if we evaluate $p(y)$ at $(t-1)$ we get the equality~$p(t-1) = q(t) (t-1)^k$
for a suitable integer~$k$. Now consider the block matrix
\[ V := \begin{pmatrix}
	0 & V^u\\
	V^b & Q(1)
	\end{pmatrix} \]
with
\[ V^u = \begin{pmatrix}
	0 & 1 & 0 && & a_1\\
	0 & 0 & 1 && & a_2\\
	\vdots && \ddots & \ddots & &\vdots\\
	&&& 0 & 1    & a_{n-1}\\
	0 &&& 0 & 0    & a_n\\
	1 & 0  & \dots &0 & 0 & 0
	\end{pmatrix} \quad V^b = \begin{pmatrix}
	-1 & 0 & 0 && & 1\\
	1 & -1 & 0 && & 0\\
	0 & 1 & \ddots & \ddots & &\vdots\\
	\vdots && \ddots & -1 & 0    & 0\\
	0 &&& 1 & -1    & 0\\
	a_1 & a_2 & \dots & a_{n-1} & a_{n} & 0
	\end{pmatrix}
 \]
and \[ Q(y) = \begin{pmatrix} 	0 	& \ldots & 0 & 0\\
			  	\vdots 	& \ddots & \vdots & \vdots\\
				0 	& \dots & 0 	  & 0\\
				0 	& \dots	 & 0 	 & y
\end{pmatrix}.
\]
The matrix~$V$ is the Seifert matrix of a link as $V-V^\transpose$ is the intersection form
of a genus $n$ oriented surface with three boundary components. Let $L$ be such a link, necessarily a $3$-component link.
We remark in passing that the matrix~$V$ from Example~\ref{example:SingleJump} is not a special case of the matrix $V$ defined in the current proof, although it is close to being so.

Recall that $B(z) = (1-z) V + (1-\ol z) V^\transpose = (\ol z-1) \cdot \left( zV -  V^\transpose \right)$.
The matrix~$V$ was constructed in such a way that
\[ B(z) =  \begin{pmatrix} 0& (\ol z-1) \cdot P(z-1)\\ (z-1) \cdot P^\transpose(\ol z-1) &  Q(-\ol z - z+2) \end{pmatrix}. \]

Using the transformations associated to the above row and column operations, we see that
$B(z)$ is congruent to
\[ B(z) \sim \begin{pmatrix} 0& (\ol z-1) \cdot A(z-1)\\ (z-1) \cdot A^\transpose(\ol z-1) &  Q(-\ol z - z+2) \end{pmatrix}. \]
Note that the matrix $Q$ is unchanged by this congruency, because in the corresponding sequence of row and column operations, it never happens that the last row or column is added to another row or column.

We complete the proof of the proposition by showing that indeed the link~$L$ has the required properties.
If $z\in S^1 \sm \{1\}$ is not a zero of $q(t)$, then also $p(z) \neq 0$. Consequently,
the form~$B(z)$ is nonsingular and metabolic. On the other hand, if $z \in S^1 \sm \{1\}$
is a root of $q(t)$, then also $p(z) = 0$. In this case the Levine-Tristram form $B(z)$ is a sum
\[ B(z) = M \oplus \begin{pmatrix} 0 & 0\\0 & -\ol z - z+2 \end{pmatrix} \]
with $M$ nonsingular and metabolic. Thus $\sigma_L(z) = 1$.
\end{proof}

\begin{remark}
Replace $Q(1)$ with $Q(0)$ in the construction of the matrix $V$ in the proof of
Proposition~\ref{prop:UniversalExample}, to obtain a matrix $V'$.
Using the same construction as in Example~\ref{example:SingleJump}, the matrices $V$ and $V'$
give rise to links~$L$ and $L'$ respectively, such that
\[ \eta_L(z) = \eta_{L'}(z) \text{ and } \sigma_L(z) = \sigma_{L'}(z)\]
for every $z \in S^1$ that is not a root of $q(t)$.
Analogously to Example~\ref{example:SingleJump}, $L$ and $L'$ are not concordant, but again this
can also be seen using linking numbers.  This leads to the following question.
Does there exist a pair of links $L$ and $L'$, with the same pairwise linking numbers,
whose signature and nullity
functions can only tell the concordance classes of the links apart at an
isolated algebraic numbers $z, \bar z \in S^1$ that are roots of the Alexander
polynomial $\Delta_L = \Delta_{L'}$.

%have the same properties, but the links have the same pairwise linking numbers?
\end{remark}

\section{Twisted homology and integral homology isomorphisms}\label{section:twisted-homology-key-lemma}

Now we begin working towards the proof of Theorem~\ref{theorem:main-theorem}.
Fix $z \in S^1 \sm\{1\}$ to be a unit complex number that is not the root of any
polynomial $p(t) \in \Z[t]$ with $p(1)=\pm 1$ i.e.\ $z$ is not a
Knotennullstelle.  We denote the classifying space for the integers~$\Z$ by $B\Z$, which has
the homotopy type of the circle~$S^1$.
Given a CW complex $X$, a map $X \to B\Z$ induces a homomorphism
$\pi_1(X) \to \Z$.  This determines a representation
\[\a \colon \Z[\pi_1(X)] \to \Z[\Z] \xrightarrow{\operatorname{ev}_z} \C\]
of the group ring of the fundamental group of $X$,
 with respect to which we can consider the twisted homology
 \[H_i(X;\C^\a) := H_i\left(\C \otimes_{\Z[\pi_1(X)]} C_*(\wt{X}) \right).\]

Let $\Sigma \subset \Z[\Z]$ be the multiplicative subset of polynomials that
map to $\pm 1$ under the augmentation $\eps \colon \Z[\Z] \to \Z$, that is
$\Sigma = \{p(t) \in \Z[\Z] \, |\, |p(1)|=1 \}$. By inverting this subset we obtain
the localisation~$\Sigma^{-1}\Z[\Z]$ of the Laurent polynomial ring.
This has the following properties.
\begin{enumerate}[(i)]
  \item The canonical map $\Z[\Z] \to \Sigma^{-1}\Z[\Z]$ is an inclusion, since $\Z[\Z]$ is an integral domain.
  \item\label{item:property-cohn-local-two} For any $\Z[\Z]$-module morphism $f \colon M \to N$ of finitely generated free $\Z[\Z]$-modules such that the augmentation \[\eps(f) = \Id \otimes f \colon \Z \otimes_{\Z[\Z]} M \to \Z \otimes_{\Z[\Z]} N\] is an isomorphism, we have that \[\Id \otimes f \colon \Sigma^{-1}\Z[\Z] \otimes_{\Z[\Z]} M \to \Sigma^{-1}\Z[\Z] \otimes_{\Z[\Z]} N\] is also an isomorphism.
\end{enumerate}

The second property can be reduced to the following.
Assume $A$ is a matrix over $\Z[\Z]$ such that $\eps(A)$ is invertible. Consequently,
we have $\det( \eps(A)) = \pm 1$ and as $\eps( \det(A) ) = \det( \eps(A))$, we deduce that
$\det(A) \in \Sigma$. Therefore, the determinant~$\det(A)$ is invertible in the localisation~$\Sigma^{-1}\Z[\Z]$ and so is the matrix~$A$ over $\Sigma^{-1}\Z[\Z]$.

As the unit modulus complex number~$z$ that we have fixed is not a Knotennullstelle, the
representation $\a$ defined above factors through the localisation,
i.e.\ evaluation at~$z$ determines a ring homomorphism $\Sigma^{-1}\Z[\Z] \xrightarrow{\Sigma^{-1}\operatorname{ev}_z} \C$
such that the ring homomorphisms $\Z[\Z] \xrightarrow{\operatorname{ev}_z} \C$ and
\[\Z[\Z] \to \Sigma^{-1}\Z[\Z] \xrightarrow{\Sigma^{-1}\operatorname{ev}_z} \C\] coincide.

\begin{lemma}\label{lemma:Z-hom-iso-implies-twisted-iso}
  Let $f \colon X \to Y$ be a map of finite CW complexes over $S^1$, that is there are maps $g \colon X \to S^1$ and $h \colon Y \to S^1$ such that $h \circ f = g$, and suppose that
  \[f_* \colon H_i(X;\Z) \toiso H_i(Y;\Z)\]
  is an isomorphism for all $i$.  Then
  \[f_* \colon H_i(X;\C^\a) \toiso H_i(Y;\C^\a)\]
is also an isomorphism for all $i$.
\end{lemma}

The lemma follows~\cite[Proposition~2.10]{Cochran-Orr-Teichner:1999-1}.
The difference is that we use the well-known refinement that one does not need
to invert all nonzero elements.  We give the proof for the convenience of the
reader.  This is adapted from the proof given in~\cite{Friedl-Powell:2010-1}.

\begin{proof}
The algebraic mapping cone $D_*:= \mathscr{C}(f_* \colon C_*(X;\Z) \to C_*(Y;\Z))$ has vanishing homology, and comprises finitely generated free $\Z$-modules.  Therefore it is chain contractible.  We claim that the chain contraction can be lifted to a chain contraction for $\mathscr{C}(f_* \colon C_*(X;\Sigma^{-1}\Z[\Z]) \to C_*(Y;\Sigma^{-1}\Z[\Z]))$, the mapping cone over the localisation $\Sigma^{-1}\Z[\Z]$.

To see this, let $s \colon D_* \to D_{*+1}$ be a chain contraction, that is we have that $\partial s_i + s_{i-1}\partial  = \Id_{D_i}$ for each~$i$.
Define $\wt{D}_*:= \mathscr{C}(f_* \colon C_*(X;\Z[\Z]) \to C_*(Y;\Z[\Z]))$ and consider
$\eps \colon \wt{D}_* \to D_* = \Z \otimes_{\Z[\Z]} \wt{D}_*$, induced
by the augmentation map. Denote
$E_* := \mathscr{C}(f_* \colon C_*(X;\Sigma^{-1}\Z[\Z]) \to C_*(Y;\Sigma^{-1}\Z[\Z]))$ and note that there is an inclusion $\wt{D}_i \to E_i = \Sigma^{-1}\Z[\Z] \otimes_{\Z[\Z]} \wt{D}_i$, induced by the localisation.
Lift $s$ to a map $\wt{s} \colon \wt{D}_{*} \to \wt{D}_{*+1}$, as in the next diagram
\[\xymatrix  @C+1cm {\wt{D}_* \ar@{-->}[r]^{\wt{s}} \ar[d]^{\eps} & \wt{D}_{*+1} \ar[d]^{\eps} \\ D_* \ar[r]^{s} & D_{*+1}.
}\]
The lifts exist since all modules are free and $\eps$ is surjective. But then we have that
\[ f:= d\wt{s} + \wt{s}d \colon \wt{D}_* \to \wt{D}_{*}\]
is a morphism of free $\Z[\Z]$-modules whose augmentation $\eps(f)$ is an isomorphism.  Thus by property (\ref{item:property-cohn-local-two}) of $\Sigma^{-1}\Z[\Z]$,  $f$ is also an isomorphism over $\Sigma^{-1}\Z[\Z]$, and so $\wt{s}$ determines a chain contraction for $E_*$.
We therefore have that $E_* = C_*(Y,X;\Sigma^{-1}\Z[\Z]) \simeq 0$ as claimed.

Next, tensor $E_*$ with $\C$ over the representation $\a$, to get that \[\C^\a \otimes_{\Sigma^{-1}\Z[\Z]} C_*(Y,X;\Sigma^{-1}\Z[\Z]) = C_*(Y,X;\C^\a) \simeq 0.\]
Thus $H_i(Y,X;\C^\a)=0$ for all $i$ and so $f_* \colon H_i(X;\C^\a) \toiso H_i(Y;\C^\a)$
is an isomorphism for all $i$ as desired.
\end{proof}

\section{Concordance invariance of the nullity}\label{section:concordance-invariance-nullity}

In this section we show concordance invariance of the nullity function away from the set of Knotennullstellen.

\begin{definition}[Homology cobordism]
A cobordism $(W^{n+1};M^n,N^n)$ between $n$-manifolds $M$ and $N$ is said to be a \emph{$\Z$-homology cobordism}
if the inclusion induced maps $H_i(M;\Z) \to H_i(W;\Z)$ and $H_i(N;\Z) \to H_i(W;\Z)$ are isomorphisms for all~$i \in \Z$.
\end{definition}

\begin{theorem}\label{theorem:nullity-invariance}
Suppose that oriented $m$-component links $L$ and $J$ are concordant and
that $z \in S^1 \sm \{1\}$ is not a Knotennullstelle.
Then $\eta_L(z) = \eta_J(z)$.
\end{theorem}

\begin{proof}
As in the statement suppose that $z \in S^1 \sm \{1\}$ is not a Knotennullstelle.
Denote the exterior of the link~$L$ by $X_L := S^3 \sm \nu L$.  As above,
let $V$ be a matrix representing the Seifert form of $L$ with respect to a
Seifert surface $F$ and a basis for $H_1(F;\Z)$.

We assert that the matrix $zV-V^\transpose$ presents the homology
$H_1(X_L;\C^\a)$.  This can be seen as follows.  Consider the infinite cyclic cover $\ol{X}_L$ corresponding to the kernel of the homomorphism $\pi_1(X_L) \to \Z$, defined as the composition of the abelianisation $\pi_1(X_L) \to H_1(X_L;\Z) \cong \Z^m$, followed by the map $(x_1,\dots,x_m) \mapsto \sum_{i=1}^m x_i$ i.e.\ each oriented meridian is sent to $1 \in \Z$.
A decomposition of $\ol{X}_L$ and the associated Mayer-Vietoris sequence~\cite[Theorem 6.5]{Lickorish:1997-1}
give rise the following presentation
\[  \C[t^{\pm 1}] \otimes_\C H_1(F;\C) \xrightarrow{ tV - V^\transpose }
\C[t^{\pm 1}] \otimes_\C H_1(F;\C)^\vee \rightarrow   H_1( \ol{X}_L ;\C) \rightarrow 0, \]
where $H_1(F;\C)^\vee$ is the dual module $\Hom_{\C}(H_1(F;\C),\C)$.
Apply the right-exact functor~$\C^\alpha \otimes_{\C[t^{\pm 1}]}$ to this sequence, to obtain the sequence
\[  \C^{\a} \otimes_\C H_1(F;\C) \xrightarrow{ zV - V^\transpose }
\C^{\a} \otimes_\C H_1(F;\C)^\vee \rightarrow  \C^\alpha \otimes_{\C[t^{\pm 1}]} H_1(\ol{X}_L ;\C)
\rightarrow 0. \]
As $H_0(\ol{X}_L; \C) \cong \C$, we have that $\Tor_1^{\C[t^{\pm 1}]}(H_0(\ol{X}_L;\C),\C^{\a})=0$ by
the projective resolution
\[ 0 \to \C[t^{\pm 1}] \overset{\cdot (1-t)}{\to} \C[t^{\pm 1}]\to \C \to 0 \]
and $z \neq 1$.
Since $\C[t^{\pm 1}]$ is a principal ideal domain, we can apply the universal coefficient theorem for homology to deduce that $\C^\alpha \otimes_{\C[t^{\pm 1}]} H_1(\ol{X}_L ;\C) = H_1(X_L; \C^\alpha)$.  This completes the proof of the assertion that $zV-V^\transpose$  presents the homology $H_1(X_L;\C^\a)$.

Next observe that $(\ol{z} - 1)(zV-V^\transpose) = (1-z)V + (1-\ol{z})V^\transpose$
 presents the same module as $zV-V^\transpose$, since $\ol{z}-1$ is nonzero.  The dimension of $H_1(X_L;\C^\a)$ therefore coincides with the nullity $\eta_L(z)$, which is by definition the nullity of the matrix $(1-z)V + (1-\ol{z})V^{\transpose}$.

Now, let $A \subset S^3 \times I$ be a union of annuli giving a concordance
between $L$ and $J$, and let  $W := S^3 \times I \sm \nu A$.  Then $W$ is a
$\Z$-homology bordism between $X_L$ and $X_J$; this is a straightforward
computation with Mayer-Vietoris sequences or with Alexander duality; see for example \cite[Lemma~2.4]{Friedl-Powell-2014}.
Thus by two
applications of Lemma~\ref{lemma:Z-hom-iso-implies-twisted-iso}, with $Y= W$
and $X=X_L$ and $X=X_J$ respectively, we see that
$H_1(X_L;\C^\a) \cong H_1(W;\C^\a)\cong  H_1(X_J;\C^\a)$, and so the nullities of $L$ and $J$ agree.
We need that $z$ is not a Knotennullstelle in order to apply Lemma~\ref{lemma:Z-hom-iso-implies-twisted-iso}.
\end{proof}

\section{Identification of the signature with the signature of a 4-manifold}\label{section:identification-signatures}

In the proof of Theorem~\ref{theorem:nullity-invariance}, a key step was to reexpress
the nullity~$\eta(z)$ of the form~$B(z)$ as a topological invariant of a $3$-manifold,
and then to use the bordism constructed from a concordance to relate the invariants.
An analogous approach is used here to obtain the corresponding statement for the signature.
Everything in this section is independent of whether $z$ is a Knotennullstelle.

Recall that we fixed an oriented $m$-component link~$L \subset S^3$, and that we picked a connected Seifert surface~$F$ for~$L$.
Denote the link complement by $X_L := S^3 \sm \nu L$.
First note that the fundamental class~$[F] \in H_2(F,\partial F;\Z)$ of the Seifert surface~$F$ is independent of the choice of $F$.
This follows from the fact that its Poincar\'{e} dual is characterised as the unique cohomology
class~$\xi \in H^1(X_L;\Z)$ mapping each meridian~$\mu$ to $\xi(\mu) = 1$.

The boundary of $F \subset S^3 \sm \nu L$ is a collection of embedded curves
in the boundary tori that we refer to as the \emph{attaching curves}.
The attaching curves together with the meridians determine a framing of each boundary torus of $X_L$.
Also, this framing depends solely on $[F]$, since the connecting homomorphism of the pair
$(X_L, \partial X_L)$ maps $\partial [F] = [\partial F]$.

With respect to this framing, we can consider the Dehn filling of slope zero, resulting in the closed
$3$-manifold~$M_L$. By definition, to obtain $M_L$ attach  a disc to each of the attaching curves, and then
afterwards fill each of the resulting boundary spheres with a $3$-ball.

\begin{definition}
The framing of the boundary tori of $X_L$ constructed above is called
the \emph{Seifert framing}. The \emph{Seifert surgery} on $L$ is the $3$-manifold~$M_L$
constructed above.
\end{definition}

\begin{remark}
For links there is no reason for this framing to agree with the zero-framing of each individual component.
\end{remark}

Collapsing the complement of a tubular neighbourhood
of the Seifert surface~$F$ gives rise to map $S^3 \setminus \nu L \to S^1=B\Z$, which extends
to a map from the Seifert surgery $\phi \colon M_L \to B\Z$.
To see this in more detail, parametrise a regular neighbourhood of $F$ as $F
\times [-1,1]$, with $F$ as $F \times \{0\}$.  The intersection of this
parametrised neighbourhood with each component of $\partial F$ determines a
parametrised subset
$S^1 \times [-1,1] \subset S^1 \times S^1 \subseteq \partial F$.
Extend this to a subset $D^2 \times [-1,1] \subset D^2 \times
S^1$ for each of the Dehn filling solid tori $D^2 \times S^1$ in~$M_L$.  Now
define
\[\ba{rcl}
\phi \colon M_L &\to & S^1 = B\Z \\
x &\mapsto & \begin{cases}
  e^{\pi i t} & x= (f,t) \in \big( F \cup \bigsqcup^m D^2 \big) \times [-1,1]
\\ -1 & \text{ otherwise.}
\end{cases}
\ea\]

The map~$\phi$ classifies the image of the fundamental class of the capped-off Seifert
surface in $M_L$, in the sense that $[\phi]$ maps to $[F \cup \bigsqcup^m D^2]$
under $[M_L,S^1] \toiso H^1(M_L;\Z) \toiso H_2(M_L;\Z)$. Recall that
the homology class~$[F \cup \bigsqcup^m D^2] \in H_2(M_L; \Z)$
only depends on the isotopy class of $L$ and so also
the homotopy class of $\phi$
does not depend on the Seifert surface~$F$.
The manifold~$M_L$ together with the map~$\phi$ defines an element~$[(M_L, \phi)] \in \Omega_3(B\Z)$,
where $\Omega_k(X)$ denotes the bordism group of oriented,
topological $k$-dimensional manifolds with a map to $X$.
Recall that cobordism is a generalised homology theory
fulfilling the suspension axiom, see e.g.\ \cite[Chapter 21]{tomDieck08} and \cite[Section 14.4]{99May}.
As a consequence, we obtain
\[\wt{\Omega}_3(B\Z) = \wt{\Omega}_3(S^1) = \wt{\Omega}_3(\Sigma S^0) \cong \wt{\Omega}_2(S^0) = \Omega_2(\pt) =0.\]
Thus $\Omega_3(B\Z) \cong \Omega_3(\pt) = 0$ \cite{Rohlin53}.

The group $\Omega_3(B\Z) \cong \Omega_3 \oplus \Omega_2 = 0 \oplus 0 = 0$ is trivial,
and we can make use of this fact to define a signature defect invariant, as follows.

For any oriented $3$-manifold $M$ with a map $\phi \colon M \to B\Z$, we will define an
integer for each complex number $z \in S^1$.
Since $\Omega_3(B\Z) =0$, there exists a
$4$-manifold $W$ with boundary $M$ and a map $\Phi \colon W \to B\Z$ extending
the map $M \to B\Z$ on the boundary.  Similarly to before, an element $z \in S^1$ determines a representation
\[\a \colon \Z[\pi_1(W)] \xrightarrow{\Phi} \Z[\Z] \xrightarrow{t \mapsto z} \C.\]
Consider the twisted homology
$H_i(W;\C^\a)$, and consider the intersection form $\lambda_{\a}(W)$ on
the quotient $H_2(W;\C^\a)/\im H_2(M;\C^\a)$.
Define the promised integer
\[\sigma(M,\phi,z) := \sigma(\lambda_{\a}(W)) - \sigma(W),\]
where $\sigma(W)$ is the ordinary signature of the intersection form on $W$.

The proof of the following proposition is known for the coefficient system~$\Q(t)$, e.g.\ ~\cite{Powell-2016-signatures-4-genus}.
For the convenience of the reader, we sketch the key steps for an adaptation to $\C^\alpha$.

\begin{proposition}~
  \begin{enumerate}[(i)]
    \item The intersection form $\lambda_{\a}(W)$ is nonsingular.
    \item The signature defect $\sigma(M,\phi,z)$ is independent of the choice of $4$-manifold~$W$.
  \end{enumerate}
\end{proposition}

\begin{proof}
The long exact sequence of the pair~$(W, \partial W)=(W,M)$ gives rise to the following commutative diagram
\[\xymatrix{
\ldots \ar[r] & H_2(\partial W; \C^\alpha) \ar[r] & H_2(W; \C^\alpha) \ar[r] \ar[rdd]
	&  H_2(W, \partial W; \C^\alpha) \ar[d]^{\PD_W^{-1}} \ar[r] & \ldots \\
&&& H^2(W;\C^\alpha) \arrow[d]^{\kappa} & \\
&&& \left( H_2(W;\C^\alpha) \right)^\vee, &
}\]
where for a $\C$-module $P$ we denote its dual module by $P^\vee := \Hom_\C(P,\C)$.
Since Poincar\'{e}-Lefschetz duality~$\PD_W$ and the Kronecker pairing~$\kappa$ are isomorphisms,
we obtain an injective map $H_2(W;\C^\a)/ \im H_2(M;\C^\a) \to H_2(W;\C^\a)^\vee$.
This map descends to
\[ \lambda_\a \colon H_2(W;\C^\a)/ \im H_2(M;\C^\a)
	\to \left( H_2(W;\C^\a) / \im H_2(\partial W;\C^\a)\right)^\vee,\]
so that the diagram below commutes:
\[\xymatrix{
H_2(W;\C^\a)/ \im H_2(M;\C^\a) \ar[rd]_{\lambda_\a} \ar @{^{(}->}[r] &  H_2(W;\C^\a)^\vee \\
	& \left(H_2(W;\C^\a)/ \im H_2(\partial W;\C^\a) \right)^\vee. \ar[u] }.\]
Consequently, the form~$\lambda_\a$ is nondegenerate, and so it is nonsingular since
it is a form over the field~$\C$.

We proceed with the second statement of the proposition, namely independence of $\sigma(M, \phi, z)$ on the choice of $W$.
Suppose that we are given two $4$-manifolds $W^+, W^-$, both with boundary
$\partial W^\pm = M$, and a map $\Phi^\pm \colon W^\pm \to B\Z$ extending $\phi \colon M \to B\Z$.
Temporarily, define the signature defects arising from the
two choices to be
\[\sigma(W^{\pm},\Phi^{\pm},z) := \sigma(\lambda_{\a}(W^{\pm})) - \sigma(W^{\pm}).\]
We will show that $\sigma(W^+,\Phi^+,z) = \sigma(W^-,\Phi^-,z)$, and thus that $\sigma(M,\phi,z)$ is a well-defined integer, so our original notation was justified.

Glue $W^+$ and $\overline{W^-}$ together along~$M$, to obtain a closed manifold~$U$,
together with a map $\Phi \colon U \to B\Z$. By Novikov additivity, we learn that
\[ \sigma_z(U, \Phi) := \sigma(\lambda_\a(U)) - \sigma(U)
= \sigma(W^+, \Phi^+, z) - \sigma(W^-, \Phi^-, z).\]
This defect~$\sigma_z(U, \Phi)$ can be promoted to a bordism invariant
$\sigma_z \colon \Omega_4(B\Z) \to \Z$, see e.g.\ ~\cite[Proof of Lemma~3.2]{Powell-2016-signatures-4-genus}
 and replace $\Q(t)$ coefficients with $\C^\a$ coefficients.

\begin{claim}
The map $\sigma_z \colon \Omega_4(B\Z) \to \Z$ is the zero map.
\end{claim}

Let $U$ be a closed $4$-manifold together with a map $\Phi \colon U \rightarrow S^1$, representing an element of $\Omega_4(B\Z)$.
By the axioms of generalised homology theories, we have
\[\wt{\Omega}_4(S^1) = \wt{\Omega}_4(\Sigma S^0) \cong \wt{\Omega}_3(S^0) = \Omega_3(\pt) =0.\]
Thus an inclusion $\pt \to S^1$ induced an isomorphism $\Omega_4(\pt) \xrightarrow{\cong} \Omega_4(S^1)$.
%From the Atiyah-Hirzebruch spectral sequence~\cite[Example XV.7.2]{Eilenberg56}
%one can compute that $\Omega_4(\pt) \xrightarrow{\cong} \Omega_4(B\Z)$, where the map is
%induced by an inclusion $\pt \to S^1$.
So $(U,\Phi)$ is bordant over $S^1$ to a $4$-manifold $U'$ with a null-homotopic map $\Phi'$ to $S^1$. In this case the local coefficient system~$\C^\a$ is just the trivial representation~$\C$.
Consequently, we have $\lambda_\a(U') = \lambda(U')$, so $\sigma_z(U', \Phi') =0$.
By bordism invariance, $\sigma_z(U,\Phi)=0$, which completes the proof of
the claim.

Now the independence of $\sigma(M, \Phi, z)$ on the choice of $W$ follows from
\[ 0 = \sigma_z(U, \Phi) = \sigma(W^+, \Phi^+, z) - \sigma(W^-, \Phi^-, z). \]
\end{proof}

Now that we have constructed an invariant, we need to relate it to the Levine-Tristram signatures.
Recall that $L$ is an oriented link, that $M_L$ is the Seifert surgery, and that we constructed a canonical map $\phi \colon M_L \to S^1$, well-defined up to homotopy.

Let $\operatorname{Lk}_L$ be the linking matrix of the link~$L$ in the Seifert framing, that
is the entry~$(\operatorname{Lk}_L)_{ij}$ is the linking number~$\lk(L_i, L_j)$ between the components
~$L_i$ and $L_j$ if $i \neq j$, and the Seifert framing of~$L_i$ if $i=j$.
The sum $\sum_{i} [\ell_i]$ in $H_1(X_L;\Z)$ of the Seifert framed longitudes vanishes.  Note that \[[\ell_i] = \sum_{j} \lk(L_i,L_j) [\mu_j] \in H_1(X_L;\Z) \cong \Z\langle\mu_i \mid i=1,\dots,n \rangle,\]
where $\mu_i$ is a meridian of the $i$--th component of $L$. We then have
\[0 = \sum_{i} [\ell_i] =  \sum_i\sum_{j} \lk(L_i,L_j) [\mu_j] =  \sum_j\sum_{i} \lk(L_i,L_j) [\mu_j],\]
from which it follows that $\sum_{i} \lk(L_i,L_j)=0$ for every $j=1,\dots,n$. That is, the sum of the entries in each row and in each column of the matrix $\operatorname{Lk}_L$ is zero.  We will use this observation in the proof below.

\begin{lemma}\label{lemma:identification}
Suppose that $z \in S^1 \sm\{1\}$ and let $\phi \colon M_L \to S^1$ be the map defined at the beginning of this section.  Then we have
	\[\sigma(M_L,\phi,z) = \sigma_L(z) -\sigma( \operatorname{Lk}_L ).\]
\end{lemma}

\begin{proof}
Construct a $4$-manifold with boundary $M_L$ as follows.  Let $F$ be a connected Seifert surface for $L$.
Push the Seifert surface into $D^4$ and consider its complement~$V_F := D^4 \sm \nu F$.
Note that if we cap $F$ off with $m$ 2-discs, we obtain a closed surface.
Let $H$ be a 3-dimensional handlebody whose boundary is this surface.
Note that $\partial V_F = X_L \cup F \times S^1$.  Then define
\[W_F := V_F \cup_{F \times S^1} H \times S^1.\]
Note that $\partial W_F = M_L$.
By \cite[pp.~538-9]{Ko:1989-1} and \cite[Lemma~5.4]{Cochran-Orr-Teichner:2002-1},
we have that $\lambda_z(W_F) = (1-z)V + (1-\ol{z})V^\transpose$.

Now we show that~$\sigma(W_F) = \sigma(\operatorname{Lk}_L)$.
For this we use Wall's additivity formula~\cite{Wall69} for the signature.
We follow the notation of~\cite[Section 2.3]{Conway17}, and ask the reader to consult ibidem.
Consider~$W_F$ as the result of the gluing
\[ W_F = V_F \cup_{F\times S^1} H\times S^1.\]
Write~$\Sigma := \partial F \times S^1$, and observe that $H_1(\partial F \times S^1; \Q) = \Q\langle \mu_1, \ell_1,\dots, \mu_n,\ell_n \rangle$
is generated by a collection of meridians~$\mu_i$ and Seifert--framed longitudes~$\ell_i$  of the $i$--th component, where $i=1,\dots,m$.
For this we consider~$\partial F \times S^1$ as the boundary of the closure of~$\nu L \subset S^3$.
After gluing along $M := F \times S^1$, the remaining boundary is the union of $N_+ := X_L$ and $N_- := \sqcup_i D^2_i \times S^1$ along $\partial N_+ = \partial N_- = \sqcup_i S^1_i \times S^1$,
where the discs~$D^2_i$ are the complement of $F$ in~$\partial H$.
Figure~\ref{fig:Walladditivity} sketches the set-up so far.

\begin{figure}
  \begin{center}
  \includegraphics[width=10cm]{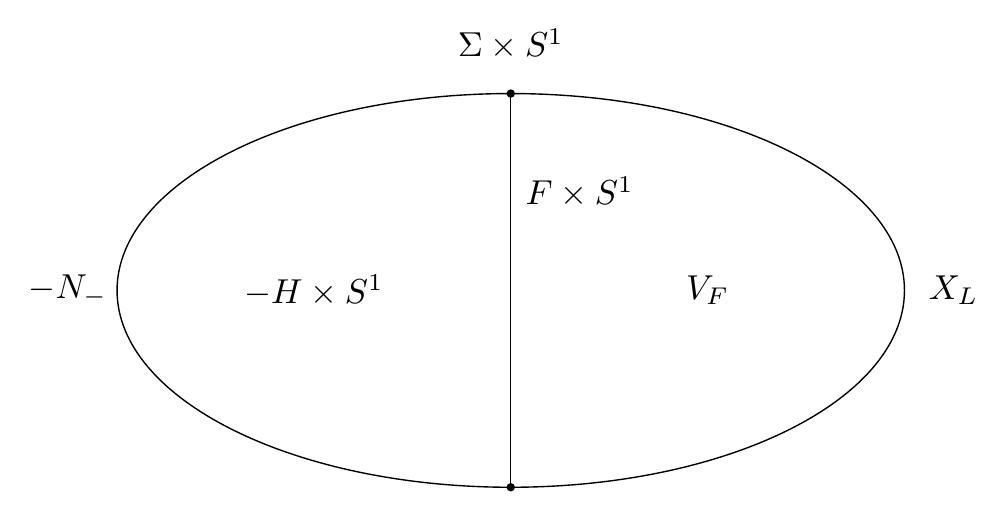}
  \end{center}
  \caption{Set-up for Wall additivity.}
  \label{fig:Walladditivity}
\end{figure}

Now, we compute the kernels \[V_X := \ker \big( H_1(\partial F \times S^1; \Q) \to H_1(X; \Q) \big)\] of the inclusions for $X = M, N_+, N_-$:
\begin{align*}
	V_{N_-} &= \langle \ell_i \mid 1\leq i \leq n \rangle,\\
	V_{N_+} &= \langle \ell_i - \sum_j \lk(L_j, L_i) \mu_j \mid 1\leq i \leq n\rangle,\\
	V_{M} &= \langle \ell_1 + \cdots + \ell_n,  \mu_i - \mu_j \mid 1\leq i< j \leq n \rangle.
\end{align*}
Our convention for computing the Maslov index is to express elements~$\alpha_i \in V_{X_L} = V_{N_+}$
as a sum~$x_i + y_i$ with $x_i \in V_{N_-}$ and $y_i \in V_{M}$, and
then consider the pairing~$\Psi(\alpha_i, \alpha_j) := x_i \cdot y_j$, where the~$\cdot$ is the skew-symmetric
intersection product on the surface~$\Sigma$, with $\mu_i \cdot \ell_j = \delta_{ij}$; c.f~\cite{RanickiMaslov}.  The Maslov index is then the signature of the pairing~$\Psi$.
So let us relate this pairing to the linking matrix.
Note that a suitable decomposition of a basis $\ell_i - \sum_j \lk(L_j, L_i) \mu_j$ for $V_{N_+}$ is as $x_i+y_i$ with $x_i:= \ell_i$ and $y_i= - \sum_j \lk(L_j, L_i) \mu_j$. Here $y_i \in V_M$ because the sum of coefficients $\sum_j \lk(L_j, L_i) =0\in \Z$, by the observation made just before the statement of the lemma, from which it follows that $y_i$ can be expressed as a linear combination of homology classes of the form $\mu_i -\mu_j$.
We then take
$\alpha_i = \ell_i + y_i \in V_{N_-} + V_{M}$,
and $\alpha_j = x_j - \sum_k \lk(L_k, L_j) \mu_k \in V_{N_-} + V_M$, and compute:
\[ \Psi(\alpha_i, \alpha_j) = -\ell_i \cdot \sum_k \lk(L_k, L_j) \mu_k = \lk(L_i, L_j).\]
The Maslov correction term is therefore~$\sigma(\Psi) = \sigma (\operatorname{Lk}_L)$. Together with~$\sigma (V_F) = 0$, this implies that~$\sigma (W_F) = \sigma (\operatorname{Lk}_L)$.

Therefore, we obtain the following equality
\[ \sigma(\lambda_z(W_F))-\sigma(W_F) = \sigma((1-z)V + (1-\ol{z})V^\transpose) - \sigma (\operatorname{Lk}_L) = \sigma(B(z)) -\sigma (\operatorname{Lk}_L) .\]
\end{proof}

\section{Concordance invariance of the signature}\label{section:conc-invariance-signature}

We start with a straightforward lemma, then we prove the final part of the main theorem.
Recall that the complement~$X_L$ and the Seifert surgery~$M_L$ are both equipped
with a homotopy class of a map to $S^1$, or equivalently with a cohomology class.
For the link complement~$X_L$, this class $\xi_L \in H^1(X_L;\Z)$
is characterised by the property that it sends each oriented meridian to $1$.

\begin{lemma}\label{lemma:concordance-implies-hom-cob-of-zero-surgeries}
Let $L$ and $J$ be concordant links.  Their Seifert surgeries $M_L$ and $M_J$ are homology bordant over $S^1$.
\end{lemma}

\begin{proof}
Denote the maps to $S^1$ by $\phi_{L} \colon M_{L} \to S^1$ and $\phi_J \colon M_L \to S^1$, and denote the corresponding cohomology
classes by $\xi_{L} \in H^1(M_{L};\Z)$ and $\xi_{J} \in H^1(M_{J};\Z)$.
Define $X_L := S^3 \sm \nu L$ and $X_J := S^3 \sm \nu J$.  Let $A \subset S^3 \times I$
be an embedding of a disjoint union of annuli giving a concordance
between $L$ and $J$.

Fix a tubular neighbourhood $\nu A = A \times D^2$ of the annulus $A$
with a trivialisation.
Denote $W_A:= S^3 \times I \sm \nu A$, whose boundary consists of the union of $X_L$, $X_J$, and a piece
identified with the total space of the unit sphere bundle~$A \times S^1$ of
$\nu A$. As usual, we refer to a representative $\{\pt\}\times S^1$ for the
$S^1$ factor in $A \times S^1$ as a \emph{meridian} of $A$. Note that the
inclusions $X_L \subset W_A$ and $X_J \subset W_A$ map the meridians in the
link complements to the meridians in $W_A$.

\begin{claim}
  There exists a cohomology class $\xi_A \in H^1(W_A;\Z)$ mapping each meridian~$\mu_A$ of $A$ to $1$.
\end{claim}
This can be seen by the Mayer-Vietoris sequence
\[ H^1(\nu A; \Z) \oplus H^1(W_A;\Z) \to H^1(\partial \nu A;\Z) \to H^2(S^3 \times I;\Z)=0, \]
in which the map $H^1(\nu A;\Z) \cong \Z^m \to H^1(\partial \nu A;\Z) \cong (\Z \oplus \Z)^m$ is given by $1 \mapsto (1,0)$ on each of the $m$ summands. That is, the homology classes of the meridians of $\partial \nu A \cong A \times S^1$ do not lie in the image of this surjective map, so they must lie in the image of $H^1(W_A;\Z)$.  This completes the proof of the claim.

It follows that $\xi_A$ is pulled back to the unique classes~$\xi_L$ and $\xi_J$ that map
the meridians in the link complements to $1$.
Using the natural isomorphism between the functors $[-, S^1]$ and $H^1(-;\Z)$,
find a map $\phi_W \colon W_A \to S^1$ that restricts to the prescribed map
$\phi_L \sqcup \phi_J \colon X_{L} \sqcup X_J \to S^1$
on the boundary.

Up to isotopy, there is a unique product structure on an annulus~$A = S^1 \times I$.
Having fixed such a structure, we consider the manifold
\[Y := W_A \cup_{A \times S^1} \bigsqcup^m (D^2 \times S^1 \times I).\]
The gluing is done in such a way as to restrict on $\bigsqcup^m S^1 \times S^1 \times \{i\}$, for $i=0,1$,
to the gluing of the Seifert surgery on $X_L$ and $X_J$. By construction, this gives a bordism
between $M_L$ and $M_J$.

Note that the map $\phi_W$ and the projection $A\times S^1 \to S^1$ glue together to give a map $\phi_Y \colon Y \to S^1$.
Equipped with this map, $(Y,\phi_Y)$ is an $S^1$-bordism between $(M_L, \phi_L)$ and $(M_J, \phi_J)$.

Finally, we assert that $Y$ is a homology bordism.
To see this, first observe, as in the proof of Theorem~\ref{theorem:nullity-invariance}, that $W_A$ is a homology bordism from $X_L$ to $X_J$.
 Flagrantly, $A \times S^1$ is a homology bordism from $S^1 \times S^1$ to itself, and $\bigsqcup^m (D^2 \times S^1 \times I)$ is a homology bordism from $\bigsqcup^m D^2 \times S^1$ to itself.
Gluing two homology bordisms together along a homology bordism, with the same maps on homology induced by the gluings
for $M_L$, $M_J$ and $Y$, it follows easily from the Mayer-Vietoris sequence
and the five lemma that $Y$ is a homology bordism.
\end{proof}

\begin{theorem}\label{theorem:signature-invariance}
Suppose that oriented $m$-component links $L$ and $J$ are concordant and that $z \in S^1 \sm \{1\}$ is not a Knotennullstelle.  Then $\sigma_L(z) = \sigma_J(z)$.
\end{theorem}

\begin{proof}
As in the statement of the theorem, suppose that $z \in S^1 \sm \{1\}$ is not a Knotennullstelle.
Let $W_{LJ}$ be a homology bordism between the Seifert surgeries $M_L$ and $M_J$,
whose existence is guaranteed by Lemma~\ref{lemma:concordance-implies-hom-cob-of-zero-surgeries}.
Let $W_J$ be a 4-manifold that gives a null-bordism of $M_J$ over $B\Z$, and define $W_L := W_{LJ} \cup_{M_J} W_J$.

The signature of the intersection form on $H_2(W_L;\C^a)/H_2(M_L;\C^\a)$, together with the ordinary signature over $\Z$, determines the signature $\sigma_L(z)$ by Section~\ref{section:identification-signatures}.  Similarly, the signature of the intersection form on the quotient $H_2(W_J;\C^\a)/H_2(M_J;\C^\a)$ and the ordinary signature of $W_J$ determine the signature $\sigma_J(z)$.
By Lemma~\ref{lemma:Z-hom-iso-implies-twisted-iso}, we have homology isomorphisms $$H_2(M_L;\C^\a) \toiso H_2(W_{LJ};\C^\a) \text{ and } H_2(M_J;\C^\a) \toiso H_2(W_{LJ};\C^\a).$$  It follows that every class in $H_2(W_L;\C^\a)$ has a representative in $W_J$, that $$H_2(W_L;\C^a)/H_2(M_L;\C^\a) \cong H_2(W_J;\C^a)/H_2(M_J;\C^\a),$$ and that this isomorphism induces an isometry of the intersection forms. Thus the twisted signatures of both intersection forms are equal.  We needed that $z$ is not a Knotennullstelle in order to apply Lemma~\ref{lemma:Z-hom-iso-implies-twisted-iso} in the preceding argument.
The same argument over $\Z$ implies that the ordinary signatures also coincide, that is $\sigma(W_L) = \sigma(W_J)$.  Therefore $\sigma(M_L,\phi_L,z) = \sigma(M_J,\phi_J,z)$.  Note that the linking number is a concordance invariant and therefore the linking matrices
agree~$\operatorname{Lk}_L = \operatorname{Lk}_{J}$.
Therefore $\sigma(M_L,\phi_L,z) + \sigma(\operatorname{Lk}_L) = \sigma(M_J,\phi_J,z) + \sigma(\operatorname{Lk}_J)$, and so $\sigma_L(z) = \sigma_J(z)$ by Lemma~\ref{lemma:identification}.  Thus the Levine-Tristram signature at $z$ is a concordance invariant, as desired.
\end{proof}

\bibliographystyle{alpha}
\bibliography{research}
\end{document}